\documentclass{amsart}
\usepackage{amsmath,amssymb,euscript,tikz,units}
\usepackage{graphicx}
\usepackage[colorlinks,citecolor=blue,linkcolor=blue,urlcolor=blue]{hyperref}
\usepackage{verbatim}
\usepackage[nameinlink]{cleveref}
\usepackage{colonequals}
\usepackage[titletoc]{appendix}
\usepackage{yhmath}
\usepackage{enumerate}
\usepackage{url}
\usepackage{enumitem}
\usepackage{fancyhdr}
\usepackage{lastpage}
\def\Vol{\mathop{\rm Area}}

\DeclareMathOperator{\Spec}{Spec}

\DeclareMathOperator{\RayQ}{RayQ}
\DeclareMathOperator{\inj}{inj}

\theoremstyle{plain}
\newtheorem{theorem}{Theorem}[section]
\newtheorem{corollary}[theorem]{Corollary}
\newtheorem{proposition}[theorem]{Proposition}
\newtheorem{lemma}[theorem]{Lemma}

\newtheorem{remark}[theorem]{Remark}

\newcommand{\be}{\begin{equation}}
\newcommand{\ene}{\end{equation}}
\newcommand{\br}{\begin{remark}}
\newcommand{\er}{\end{remark}}
\newcommand{\ben}{\begin{enumerate}}
\newcommand{\een}{\end{enumerate}}
\newcommand{\bp}{\begin{proof}}
\newcommand{\ep}{\end{proof}}
\newcommand{\beq}{\begin{equation*}}
\newcommand{\eeq}{\end{equation*}}
\newcommand{\bear}{\begin{eqnarray}}
\newcommand{\eear}{\end{eqnarray}}
\newcommand{\beqar}{\begin{eqnarray*}}
\newcommand{\eeqar}{\end{eqnarray*}}
\newcommand{\bt}{\begin{theorem}}
\newcommand{\et}{\end{theorem}}

\theoremstyle{definition}
\newtheorem{definition}[theorem]{Definition}

\theoremstyle{remark}

\newtheorem*{con*}{Construction}
\newtheorem*{rem*}{Remark}
\newtheorem*{exam*}{Example}
\newtheorem*{exams*}{Examples}
\newtheorem*{thm*}{\bf Theorem}
\newtheorem*{que*}{Question}

\newtheorem*{Def*}{Definition}
\newtheorem*{Cons*}{Construction}
\newtheorem*{Lem*}{Lemma}
\newtheorem*{Conj*}{\bf Conjecture}
\theoremstyle{plain}

\newtheorem*{def*}{Definition}

\begin{document}
\setlength{\emergencystretch}{3em}
\title[Linearly many cusps and spectral gaps]{Spectral gaps for noncompact hyperbolic surfaces with linearly many cusps}
\author{Qi Guo, Bobo Hua and Yang Shen}

\begin{abstract}
We construct complete finite-area noncompact hyperbolic surfaces with linearly many cusps and a uniform spectral gap. More precisely, for every \(\theta>0\), we construct a sequence \(S_{g,n(g)}\in\mathcal{M}_{g,n(g)}\) such that \(\lim\limits_{g\to\infty}\frac{n(g)}{g}=\theta\) and the spectrum of the Laplacian has a uniform gap above zero. The construction is based on explicit expanding \((1,3)\)-graphs, viewed as combinatorial skeletons for pants decompositions. We also establish a Steklov-type upper bound showing that expansion cannot persist when the number of boundary vertices is much larger than the genus.
\end{abstract}

\maketitle
\numberwithin{equation}{section}
\tableofcontents

\section{Introduction}

Let \(\mathcal{M}_{g,n}\) denote the moduli space of complete finite-area hyperbolic surfaces of genus \(g\) with \(n\) cusps. If \(X\in\mathcal{M}_{g,n}\) is connected, then the eigenvalue \(0\) of \(\Delta_X\) is simple, the absolutely continuous spectrum is \([1/4,\infty)\), and only finitely many discrete eigenvalues can lie in \((0,1/4)\); embedded eigenvalues may also occur in \([1/4,\infty)\), see \cite{Iwaniec02}. For \(0<c\leq 1/4\), we call \(X\) a \(c\)-\emph{expander surface} if
\[
    \Spec(\Delta_X)\cap(0,c)=\emptyset.
\]
In this paper, we study whether such a spectrum-free interval can be uniform along sequences for which the number of cusps is proportional to the genus.

The value \(1/4\) also occurs in Selberg's eigenvalue conjecture, which predicts that a congruence surface has no nonzero \(L^2\)-eigenvalue below \(1/4\); Selberg proved the lower bound \(3/16\) \cite{Sel65,Iwaniec02}. Before turning to cusps, we briefly recall that spectral gap estimates for closed random hyperbolic surfaces have recently seen substantial progress toward the \(1/4\) threshold in three standard models: the Weil--Petersson model \cite{Mir13,WX22,LW24,AM25,HMT25}, the random-covering model \cite{MNP22,MPH25,HMTRandomCovers25}, and the Brooks--Makover model \cite{BM04,SW25}. For surfaces with cusps, the number \(n\) is a second asymptotic parameter. When \(n(g)=O(g^\alpha)\) with \(\alpha<1/2\), random surfaces in the moduli space $\mathcal{M}_{g,n(g)}$
  admit a uniform spectral gap with high probability, and explicit lower bounds depending on $\alpha$ are known in this range \cite{Hide23,HWX26}. In contrast, when $n(g)$ grows significantly faster than $\sqrt{g}$
 , the spectral behavior becomes qualitatively different: arbitrarily small spectral gaps can arise in certain sublinear many‑cusp regimes \cite{SW23}, and random punctured spheres have been shown to have many small eigenvalues \cite{HT24a}. Zograf's upper bound implies that a uniform positive spectral gap is impossible when \(\lim\limits_{g\to\infty}n(g)/g=\infty\) \cite{Zo84}. In this article, we study the case that
\[
    \lim\limits_{g\to\infty}\frac{n(g)}{g}=\theta\in(0,\infty)
\]
and give a construction of hyperbolic surfaces with a uniform spectral gap.

\begin{theorem}\label{t-main-4}
For every \(\theta>0\), there exist a constant \(c_\theta>0\), a function \(n:\mathbb N\to\mathbb N\), and a sequence of complete finite-area hyperbolic surfaces
\[
    S_{g,n(g)}\in\mathcal{M}_{g,n(g)},\qquad g\geq2,
\]
such that
$$\lim\limits_{g\to\infty}\frac{n(g)}{g}=\theta$$
and, for all sufficiently large $g$,
$$\Spec(\Delta_{S_{g,n(g)}})\cap(0,c_\theta)=\emptyset.$$
Furthermore, the constant can be chosen to satisfy
\[
    c_\theta\geq c_0(1+\theta)^{-2}
\]
for a universal constant \(c_0>0\).
\end{theorem}

\begin{rem*}
   It follows from Zograf’s estimate that if $n(g)/g\to\infty$, the spectral gap of hyperbolic surface $S_{g,n(g)}\in\mathcal{M}_{g,n(g)}$
  becomes arbitrarily small. Theorem \ref{t-main-4} shows that once this condition is relaxed to $n(g)/g$ converging to a finite positive constant, a uniform positive spectral gap can in fact be obtained.
\end{rem*}

When \(\theta\) is sufficiently small, the spectrum-free interval can be chosen arbitrarily close to \((0,1/4)\). More precisely, Proposition \ref{p-theat-small} states that for every \(\varepsilon>0\), there exists \(\theta_0>0\) such that, for every \(0<\theta<\theta_0\), one can choose \(S_{g,n(g)}\in\mathcal{M}_{g,n(g)}\) with \(n(g)/g\to\theta\) and
\[
    \Spec(\Delta_{S_{g,n(g)}})\cap
    \left(0,\frac14-\varepsilon\right)=\emptyset.
\]
This argument depends on the developments in random surface theory \cite{Mir13,AM25,HMT25} and will be shown in Section \ref{sec:construction}.

\noindent\textbf{Strategy of the proof of Theorem \ref{t-main-4}.}
The proof has a discrete step and a geometric step. The main new ingredient is not the use of an expander graph as a combinatorial skeleton, which is classical, but a quantitative procedure that modifies both the boundary size and the topology while retaining expansion. Starting from explicit cubic expanders \cite{Chiu92,MSS15,Kowalski19}, we replace each edge by a fixed tree \(T_k\). This creates degree-\(1\) vertices, hence cusps in the associated surface, at a quantitatively controlled loss of the Cheeger constant. We then add loops at selected degree-\(1\) vertices. This second operation changes the genus and the number of boundary vertices but leaves every edge cut unchanged. The two operations therefore allow us to change the ratio \(n(g)/g\) to an arbitrary prescribed value \(\theta>0\), rather than only to the discrete ratios arising from a fixed replacement, while preserving a uniform expansion estimate. The relevant Cheeger estimate is proved in Lemma \ref{l-5.1}, and the resulting family \(H(g)\), satisfying
\[
    \frac{n(g)}{g}\longrightarrow\theta,
    \qquad
    \lambda_1(H(g))\gtrsim (1+\theta)^{-2},
\]
is constructed in Proposition \ref{t-main-3} of Section \ref{sec:explicit-graphs}. For the geometric step, we use the pants-skeleton construction originating in the work of Buser \cite{MR0505920} and Brooks--Makover \cite{BM04}; related discrete models in large-genus hyperbolic geometry appear in \cite{BCP21,BL21,bcl23}. The point that must be established in the present noncompact setting is that graph expansion survives the passage to a surface with cusps. This is provided by the uniform comparison
\[
    h(X_a(G))\geq C_a\min\{h(G),1\},
\]
proved in Lemma \ref{l-com-h} of Section \ref{sec:graph-to-surface}. Applying this comparison to \(H(g)\), and then using Cheeger's inequality \cite{Cheeger70,Bu82} together with Theorem \ref{l-spec}, produces the required spectral gap. Keeping track of the constants throughout the two steps yields the quantitative bound
\[
    c_\theta\geq c_0(1+\theta)^{-2};
\]
the argument is completed in Section \ref{sec:surfaces}.

In this paper, we denote by \(\mathcal{F}_{g,n}\) the class of connected \((1,3)\)-graphs with \(2g-2+n\) interior vertices of degree-$3$ and \(n\) distinguished boundary vertices of degree-$1$, see Definition~\ref{def:Fgn}. 
The preceding construction produces expander families throughout the linear regime \(n(g)\asymp g\). This regime is optimal at the level of growth rates. Indeed, if the number of boundary vertices grows faster than the genus, so that \(n(g)/g\to\infty\), then the spectral gap necessarily tends to zero; consequently, no expander family can exist in this graph class. The following estimate makes this obstruction quantitative through the discrete Steklov spectrum. Let \(\sigma_1(G)\) denote the first nonzero Steklov eigenvalue of \(G\) with boundary \(\delta G\).

\begin{theorem}\label{t-main-1}
For integers $n\geq 2,\ g\geq 0$, and connected graph \(G\in\mathcal{F}_{g,n}\), we have
\[
    \lambda_1(G)\leq\sigma_1(G)\leq\frac{16(g+1)}{3n}.
\]
\end{theorem}

Consequently, if \(n(g)/g\to\infty\), then \(\lambda_1(G_g)\to0\) for every sequence of connected graphs $\{G_g\}$ with $G_g\in\mathcal{F}_{g,n(g)}$ as $g\to\infty$. The graph construction for fixed \(n(g)/g\) and this upper bound for \(n(g)/g\to\infty\) give the discrete counterparts of the two geometric statements above.

The paper is organized as follows. Section \ref{sec2} defines the classes \(\mathcal{F}_{g,n}\), the graph Laplacian and Steklov spectrum, and the associated surfaces \(X(G)\); it also proves the comparison between \(h(G)\) and \(h(X(G))\) and its spectral consequence. Section \ref{sec:construction} constructs the graphs \(H(g)\), proves Theorem \ref{t-main-4}, and then proves Proposition \ref{p-theat-small} by puncturing suitable closed Weil--Petersson surfaces. Section \ref{sec:steklov} proves the combinatorial separation result used in the Steklov estimate, establishes Theorem \ref{t-main-1}, and derives the large-boundary corollary.

\section{\texorpdfstring{$(1,3)$}{(1,3)}-graphs and associated hyperbolic surfaces}\label{sec2}
 
We first describe the deterministic graph class used throughout the paper. We allow multiple edges, and we allow loops at degree $3$ vertices; as usual, a loop contributes $2$ to the degree of its incident vertex.

\begin{definition}\label{def:Fgn}
For non-negative integers $g$ and $n$ such that $2g-2+n>0$, $\mathcal{F}_{g,n}$ is the set consisting  of all connected graphs $G$ with a distinguished decomposition
\[
    V(G)=V^\prime(G)\sqcup \delta G
\]
satisfying the following conditions:
\begin{enumerate}[label=\textup{(\roman*)}]
    \item $|V^\prime(G)|=2g-2+n$ and $|\delta G|=n$;
    \item every vertex in $V^\prime(G)$ has degree $3$;
    \item every vertex in $\delta G$ has degree $1$.
\end{enumerate}
The vertices in $\delta G$ are called the boundary vertices of $G$.
\end{definition}

Given $G\in\mathcal{F}_{g,n}$, for any vertex $v\in V^\prime(G)$, replace it by a pair of pants. Gluing these pairs of pants according to the edges between degree $3$ vertices, and keeping the edges ending at $\delta G$ as boundary components or cusps, one obtains a surface of genus $g$ with $n$ boundary components or cusps; see Figure \ref{glue} for the case that $g=0$ and $n=4$.
    
    \begin{figure}[ht]
\centering
\includegraphics[width=0.8\textwidth]{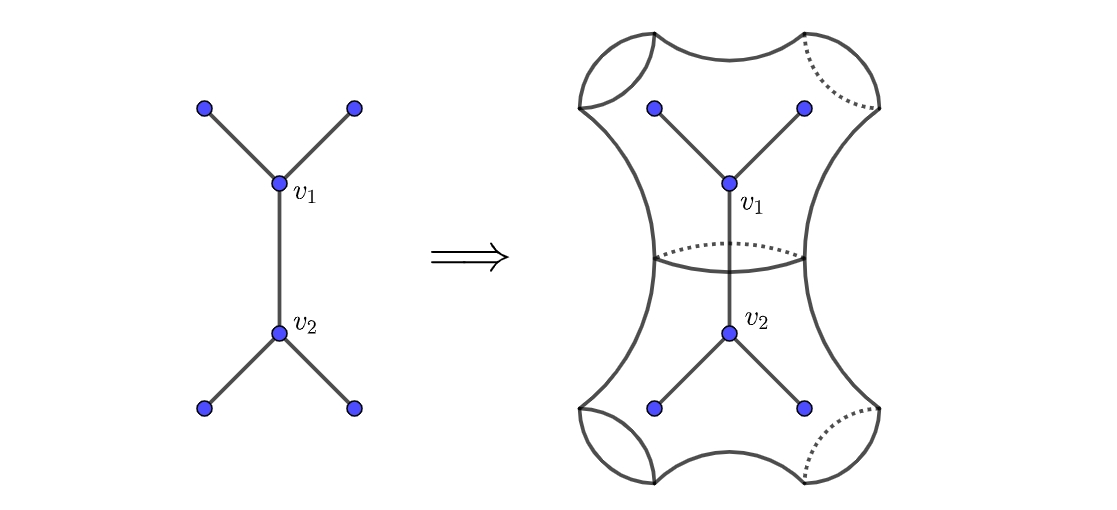}
\caption{Recovering a surface from its graphs skeleton }
\label{glue}
\end{figure}

\subsection{Cheeger constant}\label{sec-ch}
Another key geometric quantity to study expander graphs is the Cheeger constant, motivated by Cheeger \cite{Cheeger70}. The Cheeger estimate has been particularly influential, as it reveals a close relationship between the first eigenvalue of the graph Laplacian and the Cheeger constant. We recall below the graph Cheeger inequality and the Steklov comparison estimate needed in this paper.
Assume $G$ is a connected graph, and $\Omega_1,\Omega_2\subseteq V(G)$ are two subsets. Define
    $$E(\Omega_1,\Omega_2)=\left\{e\in E(G);\ \begin{matrix}\text{$e$ joins a pair of vertices}\\ \text{ from $\Omega_1$ and $\Omega_2$ respectively}\end{matrix}\right\}.$$ 
    For any subset $\Omega\subseteq V(G)$, set $$\partial\Omega=E(\Omega,\Omega^c).$$
    The Cheeger constant $h(G)$ of $G$ is defined as 
$$h(G)=\inf\limits_{\substack{\emptyset\neq \Omega\subseteq V(G),\\ |\Omega|\leq\frac{|V(G)|}{2}}}\frac{|\partial \Omega|}{|\Omega|}.$$
According to the observation of Yau, see   \cite[Theorem 8.3.6]{Buser-book}, for any hyperbolic surface $X$, $h(X)$ is realized by some curve which divides $X$ into two connected components. Similar result  holds for the case of connected graphs. We restate the result as follows. 
\begin{lemma}\cite[Theorem 2.11]{CSZ15}\label{l-con}
    For any connected graph $G$, $h(G)$ is realized by some subset $\Omega$ such that $\Omega$ and $G\setminus\Omega$ are both connected. 
\end{lemma}

\subsection{Eigenvalues on graphs}\label{subsec-eigen}
In this subsection, we introduce some definitions and properties of eigenvalues on graphs. For more details, one may refer to \cite{Gri-book,HH22}. We always assume $G$ is a connected finite graph in $\mathcal{F}_{g,n}$ for some $g\geq 0,\ n\geq 2$ such that $2g-2+n>0$. For any subset $\Omega$ of vertices, denote  
$$\mathbb{R}^{\Omega}=\{f;\ f:\Omega\to\mathbb{R}\}.$$
The Laplacian operator $\Delta_G:\mathbb{R}^{V(G)}\to\mathbb{R}^{V(G)}$ is defined as, for any $f\in\mathbb{R}^{V(G)}$ and vertex $v\in V(G)$,
$$\Delta_G f(v)=\frac{1}{\text{deg}(v)}\sum\limits_{w\sim v}\left(f(w)-f(v)\right),$$
where $w\sim v$ means that there is an edge joining $w$ and $v$. 
The eigenvalues of $-\Delta_G$ can be enumerated  as follows
$$0=\lambda_0(G)<\lambda_1(G)\leq\lambda_2(G)\leq...\leq\lambda_{\#V(G)-1}\leq 2.$$
Here $\lambda_1(G)$ is called the first Laplacian eigenvalue of the graph $G$. In this paper, a graph sequence $\{G_k = (V_k, E_k)\}$ is called an expander family if the number of vertices $|V_k|$ tends to infinity as $k \to \infty$, while the spectral gap measured by $\lambda_1(G_k)$ is uniformly bounded below; see, for instance, \cite{Kowalski19}.  Similar to the case of manifolds, the first eigenvalue is related to the Cheeger constant. In our setting, the following Cheeger's inequality holds, see e.g., \cite[Theorem 3.3]{Gri-book}. 
\begin{proposition}\label{t-ch}
    Assume $G$ is a connected graph in $\mathcal{F}_{g,n}$ for some $g\geq 0,\ n\geq 2$ with $2g-2+n>0$. Then we have 
    $$\lambda_1(G)\geq \frac{1}{18}h(G)^2.$$
\end{proposition}

We write $\delta G$ for the set of boundary vertices, i.e. the vertices of degree $1$.  The outward derivative operator is defined as
\begin{align*}
    \frac{\partial}{\partial \vec{n}}:\ &\mathbb{R}^{V(G)}\to\mathbb{R}^{\delta G}\\
    &f\to\frac{\partial f}{\partial \vec{n}},
\end{align*}
where $$\frac{\partial f}{\partial \vec{n}}(x)=f(x)-f(y)$$
for any $x\in\delta G$ and $y$ is the unique vertex such that $y\sim x$.

We introduce the Steklov problem on the pair $(G,\delta G)$, see \cite{HH22}. If a nonzero function $f\in\mathbb{R}^{V(G)}$ and $\sigma\in\mathbb{R}$ satisfy
\begin{align*}
    \begin{cases}
        \Delta_G f(x)=0,\ x\in V(G)\setminus\delta G;\\
        \frac{\partial f}{\partial \vec{n}}(x)=\sigma f(x),\ x\in\delta G,
    \end{cases}
\end{align*}
then $\sigma$ is called the Steklov eigenvalue of the graph $G$ with boundary vertices $\delta G$. For a connected graph $G$ with boundary $\delta G$, it has $|\delta G|$ Steklov eigenvalues, which can be enumerated as follows
$$0=\sigma_0(G)<\sigma_1(G)\leq...\leq \sigma_{|\delta G|-1}(G)\leq 1.$$
For any $0\not\equiv f\in \mathbb{R}^{V(G)}$ with $f|_{\delta G}\not\equiv 0$, the Rayleigh quotient is defined as 
$$R(f)=\frac{\sum\limits_{\{x,y\}\in E}(f(x)-f(y))^2}{\sum\limits_{x\in\delta G}f^2(x)}.$$
Then the first Steklov eigenvalue $\sigma_1(G)$ satisfies
\begin{align}\label{e-stek}
    \sigma_1(G)=\min\limits_{f} R(f)
\end{align}
where the minimum is taken over all nonzero functions $f$ such that $$\sum\limits_{x\in\delta G}f(x)=0\text{ and }f|_{\delta G}\not\equiv 0.$$

For any connected graph with boundary, the   Steklov eigenvalues dominate the  Laplacian eigenvalues.
\begin{theorem}\cite[Theorem 1]{SY22}\label{t-com}
    Assume $G$ is a connected graph with boundary $\delta G$. Then for any $0\leq i\leq |\delta G|-1$,
    $$\sigma_i(G)\geq\lambda_i(G).$$
\end{theorem}

\subsection{From graph skeletons to hyperbolic surfaces}\label{sec:graph-to-surface}
 We recall the comparison between the Cheeger constant and the geometric Cheeger constant of a hyperbolic surface. For related notations, see \cite{Mir13,SW23}. 

  Assume $X$ is a hyperbolic surface, denote 
 $$\mathcal{S}(X)=\left\{\alpha=\bigcup\limits_{i=1}^k\alpha_i;\begin{matrix}
     \text{For all $1\leq i\leq k$, $\alpha_i$ is a simple closed curve in $X$}\\ \text{and $X\setminus\alpha=X_1\cup X_2$, where $X_1$ and $X_2$ are}\\ \text{two disjoint subsets.}
 \end{matrix}\right\}.$$ 
 For any $\alpha\in\mathcal{S}(X)$, define
$$H(\alpha)=\frac{\ell(\alpha)}{\min\{\text{Area}(X_1),\text{Area}(X_2)\}}.$$
Then the Cheeger constant $h(X)$ of $X$ is defined as $$h(X)=\inf\limits_{\alpha\in\mathcal{S}(X)}H(\alpha).$$ 
Set
 $$\mathcal{SG}(X)=\left\{\alpha=\bigcup\limits_{i=1}^k\alpha_i;\begin{matrix}
     \text{For all $1\leq i\leq k$, $\alpha_i$ is a simple closed geodesic in $X$}\\ \text{and $X\setminus\alpha=X_1\cup X_2$, where $X_1$ and $X_2$ are}\\ \text{two disjoint connected components.}
 \end{matrix}\right\}.$$
Then the geometric Cheeger constant is defined by
$$H(X)=\inf\limits_{\alpha\in\mathcal{SG}(X)}H(\alpha).$$ 
Since for any $L>0$, there are at most finite simple closed geodesics in $X$ with lengths $\leq L$, one may check that $H(X)$ is realized by some $\alpha\in\mathcal{SG}(X)$.

Mirzakhani observed the following relationship between $h(X)$ and $H(X)$.
\begin{proposition}\cite[Proposition 4.7]{Mir13}\label{p-gh}
    Assume $X$ is a complete noncompact hyperbolic surface of finite area; then
    $$H(X)\geq h(X)\geq\frac{H(X)}{1+H(X)}.$$
\end{proposition}
\begin{remark}
    Mirzakhani   stated above proposition only for the case of compact hyperbolic surfaces. Actually the proof also works for the case of hyperbolic surfaces with some cusps. 
\end{remark}

 Assume $g,n\geq 0$ such that $2g-2+n>0$ and $G\in\mathcal{F}_{g,n}$ is a  connected graph. Recall that
 $$V^\prime(G)=\{v\in V(G);\ \textnormal{deg}(v)=3\}\text{ and }\delta G=\{v\in V(G);\ \textnormal{deg}(v)=1\}.$$
 Now we construct a hyperbolic surface $X_a(G)$ as follows.\\

\noindent \textbf{Construction}. Fix $a>0$. For any vertex $v\in V^\prime(G)$, replace it by a pair of pants $P(v)$, such that the three edges emanating from $v$ correspond to three boundary components of $P(v)$. Assume $e=v\sim w$ is an edge emanating from $v$. If $\text{deg}(w)=3$, then the corresponding boundary component is a simple closed geodesic with length $a$. If $\text{deg}(w)=1$, then the corresponding boundary component is replaced by a cusp. Gluing these pairs of pants along the corresponding simple closed geodesic without twisting, we obtain a hyperbolic surface $X_a(G)$ with genus $g$ and $n$ punctures. See Figure \ref{fig:04} for an example.

We next establish a comparison result for the Cheeger constants of $X_a(G)$.
\begin{lemma}\label{l-com-h}
    For every fixed $a>0$, there exists a constant $C=C(a)>0$ such that for any non-negative integers $g,\ n$ with $2g-2+n>0$ and graph $G\in\mathcal{F}_{g,n}$, the following inequality holds
    $$h(X_a(G))\geq C\cdot \min\{h(G),1\}.$$
\end{lemma}

\begin{proof}
    From the construction of $X_a(G)$, there is a natural pants decomposition $$X_a(G)=\bigcup\limits_{v\in V^\prime(G)}P(v).$$
    Assume $H(X_a(G))$ is realized by some simple closed multi-geodesic $\alpha\in\mathcal{SG}(X_a(G))$ and $$X_a(G)\setminus\alpha=A\cup B.$$ 
    Consider the following disjoint subsets of $V^\prime(G)$,
    $$V_1=\left\{v\in V^\prime(G);\ \textnormal{int}({P(v)})\subseteq A\right\},\ V_2=\left\{v\in V^\prime(G);\ \textnormal{int}({P(v)})\subseteq B\right\},$$
    and 
     $$V_3=\left\{v\in V^\prime(G);\text{ $\textnormal{int}({P(v)})$ intersects with $\alpha$}\right\},$$
     where $\textnormal{int}({P(v)})$ is the set of interior points of the pair of pants $P(v)$. 
   Since for any pair of pants $P$, $\textnormal{Area}(P)=2\pi$, it follows that 
   \begin{align}\label{e-area}
    \text{Area}(A)\leq 2\pi(|V_1|+|V_3|) \text{ and }\text{Area}(B)\leq 2\pi(|V_2|+|V_3|),
\end{align}
 and $\alpha$ has a decomposition 
$$\alpha=\left(\bigcup\limits_{i=1}^p\beta_i\right)\bigcup\left(\bigcup\limits_{j=1}^q\gamma_j\right)$$
   \begin{figure}[ht]
\centering
\includegraphics[width=\textwidth]{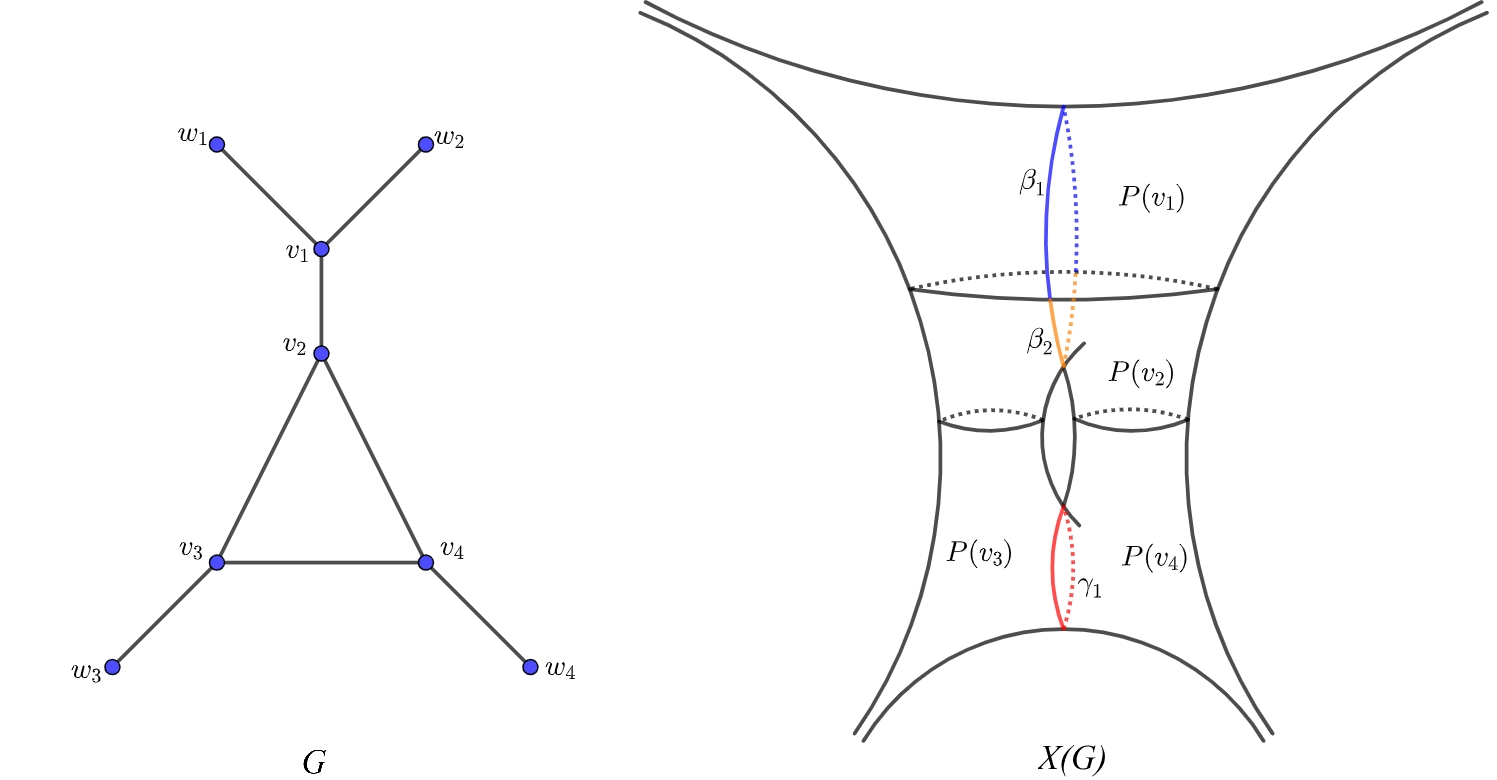}
\caption{Construction of the associated hyperbolic surface}
\label{fig:04}
\end{figure}

such that
\begin{enumerate}
    \item $\beta_i(1\leq i\leq p)$ is a simple geodesic arc contained in $P(v)$ for some $v\in V_3$ and two endpoints of $\beta_i$ are contained in the boundary geodesics of $P(v)$;
    \item $\gamma_j(1\leq j\leq q)$ is a common boundary geodesic of two pairs of pants $P(v_1)$ and $P(v_2)$ for some vertices $v_1\in V_1$ and $v_2\in V_2$.
\end{enumerate}
For example, in Figure \ref{fig:04}, we have
$$V_1=\{v_3\},\ V_2=\{v_4\},\ V_3=\{v_1,v_2\}$$
and
$$\alpha=\beta_1\cup\beta_2\cup\gamma_1.$$
There are two different types of simple geodesic arcs contained in a pair of pants:
\begin{itemize}
\item two endpoints are contained in different boundary geodesics;
\item two endpoints are contained in the same boundary geodesic . 
\end{itemize}
\begin{figure}
\centering
\includegraphics[width=0.6\textwidth]{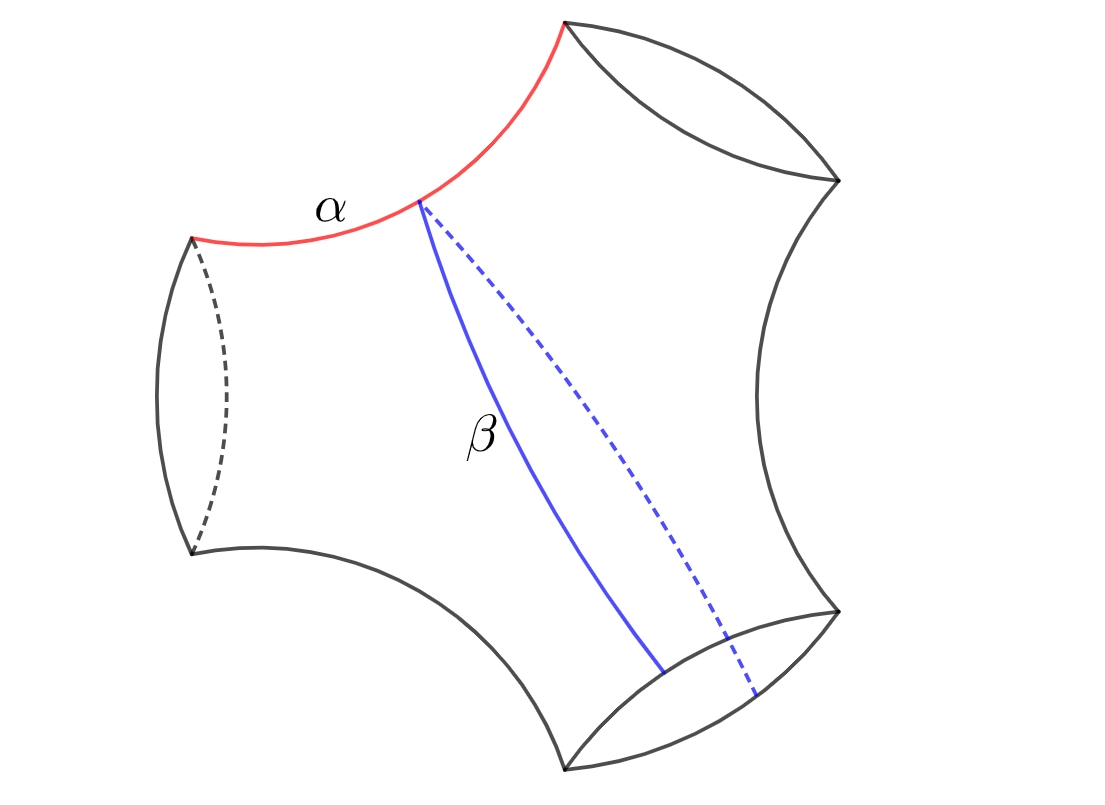}
\caption{The two essential geodesic arc types in a hyperbolic pair of pants}
\label{fig:geo-arc}
\end{figure}
Figure \ref{fig:geo-arc} illustrates the two possible arc types. It follows that there exists a constant $\ell(a)>0$ such that $\ell(\beta_i)\geq \ell(a)$ and
\begin{equation}
\begin{aligned}\label{e-length}
    \ell(\alpha)&=\sum\limits_{i=1}^p \ell(\beta_i)+\sum\limits_{j=1}^q \ell(\gamma_j)\\
    &\geq |V_3|\ell(a)+qa\geq C^\prime(a)(|V_3|+q),
\end{aligned}
\end{equation}
where $C^\prime(a)=\min\{\ell(a),a\}$ is a constant depending only on $a$. 
Note that $$|V_1|+|V_2|+|V_3|=2g-2+n.$$ If $|V_3|=2g-2+n$, then from \eqref{e-length},
\begin{align*}
H(X_a(G))&=\frac{\ell(\alpha)}{\min\{\text{Area}(A),\text{Area}(B)\}}\\
&\geq\frac{(2g-2+n)C^\prime(a)}{\pi(2g-2+n)}\\
&=\frac{C^\prime(a)}{\pi}\geq\frac{C^\prime(a)}{\pi}\min\{h(G),1\}.
\end{align*}
If $|V_3|<2g-2+n$, then one may assume $V_2\neq\emptyset$. For $i=1,2,3$, set
$$V_i^\prime=V_i\cup\left\{v;\ \begin{matrix}\textnormal{deg}(v)=1\text{ and $v$ is}\\ \text{adjacent to some vertex in $V_i$}\end{matrix}\right\}.$$
Then we have a decomposition of $V(G)$:
$$V(G)=\left(V_1^\prime\cup V_3^\prime\right)\cup V_2^\prime.$$ 
Recall that 
$$\partial V_2^\prime=\left\{e\in E(G);\ \begin{matrix}
    e\text{ is an edge joins vertices from}\\ V_2^\prime\text{ and }V_1^\prime\cup V_3^\prime
\end{matrix}\right\}.$$
For any $e\in \partial V_2^\prime$, it corresponds to an element in $V_1^\prime\cup V_3^\prime$, this element is in turn identified to a vertex in $V_3$ if it lies in $V_3^\prime$ or some simple closed geodesic $\gamma_j\ (1\leq j\leq q)$ if it lies in $V_1$. It follows that
$$|\partial V_2^\prime|\leq 3(|V_3|+q).$$
Together with the definition of the Cheeger constant, we have
\begin{align*}
3(|V_3|+q)\geq |\partial V_2^\prime|&\geq h(G)\min\{|V_1^\prime|+|V_3^\prime|,|V_2^\prime|\}\\
&\geq h(G)\min\{|V_1|+|V_3|,|V_2|\}.
\end{align*} 
If $3(|V_3|+q)\geq h(G)(|V_1|+|V_3|)$, then together with \eqref{e-area} and \eqref{e-length}, 
\begin{align*}
    H(X_a(G))&=\frac{\ell(\alpha)}{\min\{\text{Area}(A),\text{Area}(B)\}}\\
    &\geq \frac{C^\prime(a)(|V_3|+q)}{2\pi(|V_1|+|V_3|)}\geq\frac{C^\prime(a)}{6\pi}h(G).
\end{align*}
If $3(|V_3|+q)\geq h(G)|V_2|$, then together with \eqref{e-area} and \eqref{e-length},
\begin{align*}
    H(X_a(G))&=\frac{\ell(\alpha)}{\min\{\text{Area}(A),\text{Area}(B)\}}\\
    &\geq \frac{C^\prime(a)(|V_3|+q)}{2\pi(|V_2|+|V_3|)}\\
    &=\frac{C^\prime(a)}{8\pi}\cdot\frac{3(|V_3|+q)+(|V_3|+q)}{|V_2|+|V_3|}\\
    &\geq \frac{C^\prime(a)}{8\pi}\cdot\frac{h(G)|V_2|+|V_3|}{|V_2|+|V_3|}\geq\frac{C^\prime(a)}{8\pi}\min\{h(G),1\}.
\end{align*}
In summary, we have 
$$H(X_a(G))\geq\frac{C^\prime(a)}{8\pi}\min\{h(G),1\},$$
which together with Proposition \ref{p-gh} implies that
$$h(X_a(G))\geq\frac{H(X_a(G))}{1+H(X_a(G))}\geq\frac{C^\prime(a)}{8\pi+C^\prime(a)}\min\{h(G),1\}.$$
The proof is complete.
\end{proof}

\section{Construction of expander graphs and hyperbolic surfaces}\label{sec:construction}
\subsection{Explicit graph expanders in the linear-cusp regime}\label{sec:explicit-graphs}
For any $k\geq 1$, $T_k$ is a tree defined as follows
\begin{align*}
    &V(T_k)=\{v_0,v_1,...,v_{k},w_1,...,w_{k-1}\};\\
    &E(T_k)=\{v_{i}\sim v_{i+1};\ i\in[0,k-1]\}\cup\{v_i\sim w_i;\ i\in[1,k-1]\}.
\end{align*}
Moreover, $v_0$ is called the \emph{root vertex} of $T_k$. Figure \ref{fig:01} shows  the cases  $k=1,2,3$.
\begin{figure}[ht]
\centering
\includegraphics[width=0.48\textwidth]{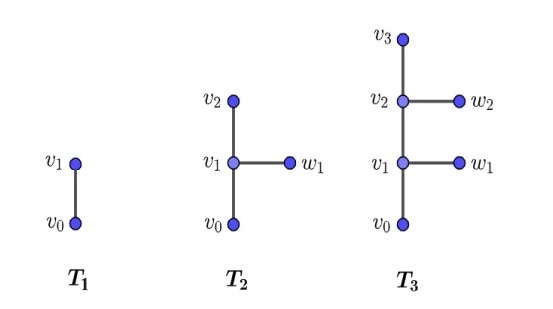}
\caption{The trees $T_k$ for $k=1, 2, 3$.}
\label{fig:01}
\end{figure}

\noindent \emph{Construction.} For any connected graph $G$ and $k\in\mathbb{N}$, the  graph $G_k$ is constructed as follows: for each $e=u_1\sim u_2\in E(G)$, where $u_1,u_2\in V(G)$, replace it by a copy of $T_k$ with root vertex $v_0$, add two edges $e_1=v_0\sim u_1$ and $e_2=v_0\sim u_2$. See Figure \ref{fig:02} for an example in the case that $k=2$.
\begin{figure}[ht]
\centering
\includegraphics[width=0.8\textwidth]{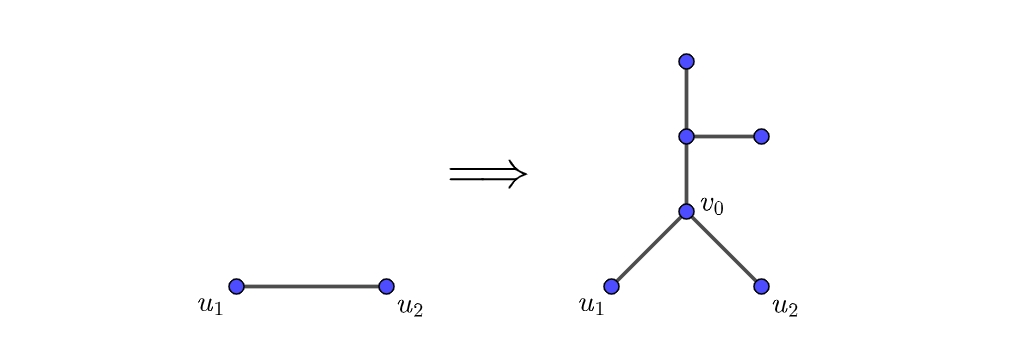}
\caption{Replacement rules for edges}
\label{fig:02}
\end{figure}
 \begin{lemma}\label{l-5.1}
       Assume $G$ is a connected $3$-regular graph and $k\in\mathbb{N}$. If $G_k$ is the graph constructed as above, then 
     $$h(G_k)\geq \min\left\{\frac{1}{2k},\frac{h(G)}{3k+1+kh(G)}\right\}.$$
 \end{lemma}
 \begin{proof}
     Assume $|V(G)|=2n$ and $|E(G)|=3n\ (n\in\mathbb{N})$. From the construction of $G_k$, there are $3n$ copies of $T_k$ contained in $G_k$ and 
     \begin{align}\label{e-5-1}
     |V(G_k)|=2n+3n\times 2k=2n+6nk.
     \end{align}
     From Lemma \ref{l-con}, there exists a subset $U_k\subseteq V(G_k)$ such that 
     \begin{enumerate}[label=(\arabic*)]
         \item $|U_k|\leq\frac{1}{2}|V(G_k)|\text{ and }
     h(G_k)=\displaystyle\frac{|\partial U_k|}{|U_k|}$;
         \item $U_k$ and $U_k^c$ are both connected.
     \end{enumerate}
     
     We first consider the case that there is a copy $T$ of $T_k$ in $G_k$ such that
      $$V(T)\not\subseteq U_k\text{ and } V(T)\not\subseteq U_k^c.$$ 
       Assume  $v_0$ is the root vertex of $T$, then we have $V(T)\setminus\{v_0\}$ is a connected component of $V(G_k)\setminus\{v_0\}$.
      If $v_0\in U_k$, then from the fact that $U_k^c$ is connected, we have either $U_k^c\subseteq V(T)\setminus\{v_0\}$ or $U_k^c\subseteq V(T)^c$. Since $V(T)\not\subseteq U_k$, it follows that $U_k^c\subseteq V(T)$ and
      $$2k=|V(T)|>|U_k^c|\geq n+3nk>2k,$$
      which leads to a contradiction. If $v_0\in U_k^c$, then  by the same arguments as above, we have $U_k\subseteq V(T)$. Hence
     $$h(G_k)=\frac{|\partial U_k|}{|U_k|}\geq\frac{1}{2k}.$$
     
     Now we assume that for any copy of $T_k$ in $G_k$, it is contained in $U_k$ or $U_k^c$.  Note that $V(G)$ could be regarded as a subset of $V(G_k)$ naturally. Set $$U=U_k\cap V(G).$$  If $U=\emptyset$, then $U_k$ is a copy of $T_k$ which implies that 
      $$h(G_k)=\frac{|\partial U_k|}{|U_k|}=\frac{2}{2k}=\frac{1}{k}.$$
      If $U=V(G)$, then $U_k^c$ is a copy of $T_k$. It contradicts the assumption that $$|U_k|\leq\frac{1}{2}|V(G_k)|.$$ 
      It remains to consider the case that $U$ and $U^c$ are both non-empty. For any edge $e=u_1\sim u_2\in \partial U$ with $u_1\in U$ and $u_2\in U^c$. Assume $T$ is the copy of $T_k$ in $G_k$ such that $u_1$ and $u_2$ are both adjacent to the root vertex $v_0\in V(T)$. From the assumption above, $T$ is contained in $U_k$ or $U_k^c$. It follows that exactly one edge $e_k$ from $\{u_1\sim v_0,\  u_2\sim v_0\}$ is contained in $\partial U_k$. One easily checks that it gives a one-to-one correspondence between $\partial U$ and $\partial U_k$, see Figure \ref{fig:03}. Hence 
      \begin{align}\label{e-5-2}
          |\partial U|=|\partial U_k|.
      \end{align}
\begin{figure}[ht]
\centering
\includegraphics[width=0.8\textwidth]{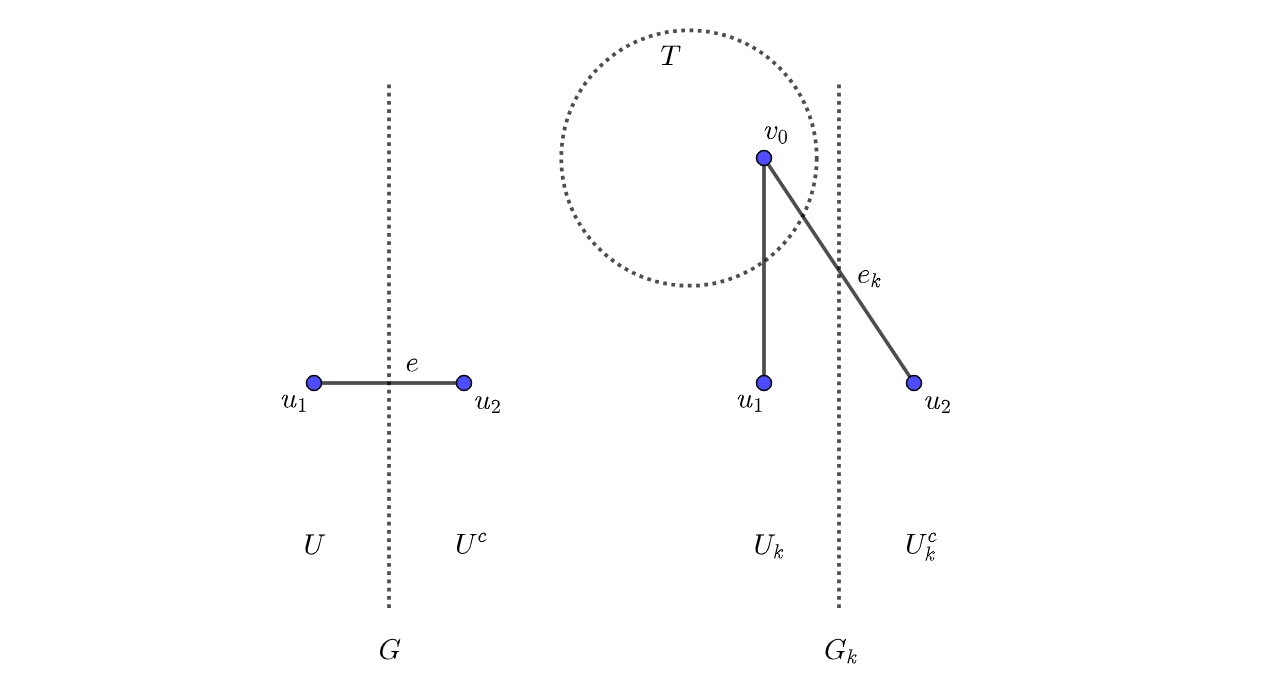}
\caption{Correspondence between $\partial U$ and $\partial U_k$}
\label{fig:03}
\end{figure}
For any edge in $\partial U$, it corresponds to a copy of $T_k$ in $G_k$. Assume that there are exactly $m\ (0\leq m\leq |\partial U|)$ of these copies contained in $U_k$. Then we have
\begin{equation}
\begin{aligned}\label{e-5-3}
|U_k|&=|U|+2k\times\frac{3|U|-|\partial U|}{2}+2km\\
&=(3k+1)|U|-k|\partial U|+2km
\end{aligned}
\end{equation}
and
\begin{align}\label{e-5-4}
    |U_k^c|=(3k+1)|U^c|-k|\partial U|+2kp
\end{align}
where $p=|\partial U|-m$. Consider the following two cases. 

\noindent \underline{\textbf{Case I}. $|U|\leq\frac{1}{2}|V(G)|=n$}. Then we have $$|\partial U|\geq h(G)\cdot|U|.$$ 
Together with \eqref{e-5-2} and \eqref{e-5-3}, it follows that
\begin{align*}
    h(G_k)=\frac{|\partial U_k|}{|U_k|}&=\frac{|\partial U|}{(3k+1)|U|-k|\partial U|+2km}\\
    &\geq\frac{|\partial U|}{(3k+1)|U|+k|\partial U|}\\
    &\geq \frac{h(G)}{3k+1+kh(G)}.
\end{align*}
\underline{\textbf{Case II}. $|U|>\frac{1}{2}|V(G)|=n$}. From \eqref{e-5-1}, $$|U_k^c|\geq \frac{1}{2}|V(G_k)|=n+3nk.$$ 
Together with \eqref{e-5-4}, it follows that
\begin{align}\label{e-5-5}
|U^c|\geq \frac{n+3nk-k|\partial U|}{3k+1}.
\end{align}
Noting that $|\partial U|\geq h(G)\cdot|U^c|$, and applying it to \eqref{e-5-5}, we have
$$|\partial U|\geq\frac{(n+3nk)h(G)}{3k+1+kh(G)}.$$
Together with \eqref{e-5-2}, it follows that $$h(G_k)=\frac{|\partial U_k|}{|U_k|}\geq\frac{|\partial U|}{n+3nk}\geq \frac{h(G)}{3k+1+kh(G)}.$$
The proof is complete.
 \end{proof}

Take a constant \(\eta_0>0\) and a family $\{G(m)\}_{m\geq 1}$ of connected \(3\)-regular graphs such that for any $m\geq 1$,
\[
    |V(G(m))|=2m
    \quad\text{and}\quad
    h(G(m))\geq \eta_0.
\]
Such a family exists from standard explicit constructions of expander graphs; see, e.g., the explicit construction of cubic Ramanujan graphs \cite{Chiu92,Kowalski19} and all-sizes construction of cubic Ramanujan multigraphs
\cite{MSS15,Cohen16}. Set
\[
    \alpha_0=\min\left\{\frac12,\frac{\eta_0}{4+\eta_0}\right\}>0.
\]
For any $k,\ m\in\mathbb{N}$, assume \(G_k(m)\) is the graph induced from $G(m)$ as in the construction above. Then it is straightforward to verify that
\begin{align}\label{e-contain}
    G_k(m)\in\mathcal{F}_{m+1,3mk}.
\end{align}
From Lemma \ref{l-5.1} and Proposition \ref{t-ch}, we have
\begin{align}\label{e-g}
    h(G_k(m))\geq \min\left\{\frac{1}{2k},\frac{\eta_0}{3k+1+k\eta_0}\right\}\geq \frac{\alpha_0}{k}
\end{align}
and 
\begin{align}\label{e-eigenvalue}
    \lambda_1(G_k(m))\geq\frac{\alpha_0^2}{18k^2}
\end{align}
Now we are ready to prove the following proposition.
\begin{proposition}\label{t-main-3}
For every \(\theta>0\), there exist a function \(n:\mathbb N\to\mathbb N\) and a family of connected graphs $\{H(g)\}_{g\geq 2}$ such that
\begin{enumerate}[label=\textup{(\roman*)}]
    \item \(H(g)\in\mathcal{F}_{g,n(g)}\);
    \item \(\displaystyle\lim_{g\to\infty}\frac{n(g)}{g}=\theta\);
    \item \(\lambda_1(H(g))\geq \frac{\alpha_0^2}{2(\theta+3)^2}\).
\end{enumerate}
\end{proposition}

\begin{proof}
    For any \(\theta>0\), choose \(k\in\mathbb{N}\) such that \(\theta\leq 3k<\theta+3\). Let $\{G_k(m)\}_{m\geq 1}$ be the sequence of connected graphs defined above. 
    
    If \(\theta=3k\), set $n(g)=3(g-1)k$ and $H(g)=G_k(g-1)$ for $g\geq 2$. Then it follows from \eqref{e-contain}, \eqref{e-g} and \eqref{e-eigenvalue} that function $n$ and graphs $\{H(g)\}_{g\geq 2}$ are desired.
    
   Now we assume \(\theta<3k<\theta+3\). Set
   $$m_0=\max\left\{\left\lceil\frac{\theta}{3k-\theta}\right\rceil,\left\lceil\frac{3k-2\theta}{(3k+1)\theta}\right\rceil\right\}$$
   For \(m\geq m_0\), define
   \[
       t_m=\left\lfloor\frac{(3k-\theta)m-\theta}{1+\theta}\right\rfloor\geq 0.
   \]
   Then we have for each $m\geq m_0$,
   \begin{equation}
   \begin{aligned}\label{e-g-1}
       t_{m+1}-t_m&\leq 1+\frac{(3k-\theta)(m+1)-\theta}{1+\theta}-\frac{(3k-\theta)m-\theta}{1+\theta}\\
       &=1+\frac{3k-\theta}{1+\theta}<4
   \end{aligned}
   \end{equation}
   and
   \begin{align}\label{e-num-11}
       t_{m+1}\leq 3km.
   \end{align}
   For $m\geq m_0$, set $g_m=m+1+t_m$. We now start to construct the desired function $n$ and sequence $\{H(g)\}_{g\geq 2}$ of graphs.
   
   If $2\leq g<g_{m_0}$, set $n(g)=3k(g-1)$ and $H(g)=G_k(g-1)\in\mathcal{F}_{g,n(g)}$. From \eqref{e-g}, \eqref{e-eigenvalue} and the assumption that $3k<\theta+3$, we have 
   $$h(H(g))\geq\frac{\alpha_0}{k}\geq\frac{3\alpha_0}{\theta+3}$$
   and
   \begin{align}\label{e-e1}
      \lambda_1(H(g))\geq\frac{\alpha_0^2}{18k^2}\geq\frac{\alpha_0^2}{2(\theta+3)^2},
   \end{align}

   If \(g_m\leq g<g_{m+1}\) for some \(m\geq m_0\), from \eqref{e-g-1}, one may write $$g=g_m+u=m+1+u+t_m$$ for some \(0\leq u< 4\). From \eqref{e-num-11}, we have $t_m+u<t_{m+1}<3mk$, it follows that one may take \(v_1,\ldots, v_{t_m+u}\in V(G_k(m))\) which are \(t_m+u\) vertices of degree-$1$. Set $n(g)=3km-t_m-u$, and let \(H(g)\) be the graph obtained from $G_k(m)$ by adding a loop at each vertex \(v_i\)  \((1\leq i\leq t_m+u)\). One may check that \(H(g)\) is a connected graph with \(3km-t_m-u\) vertices of degree-$1$ and \(3km+2m+t_m+u\) vertices of degree-$3$, hence 
   $$H(g)\in\mathcal{F}_{g,n(g)},$$
   which implies that $n$ and $\{H(g)\}_{g\geq 2}$ satisfy the condition $(i)$.
   
  From \eqref{e-g}, we have for \(g_m\leq g<g_{m+1}\ (m\geq m_0)\),
   \begin{align*}
   h(H(g))=h(G_k(m))\geq\frac{\alpha_0}{k}.
   \end{align*}
    Together with Proposition \ref{t-ch} and the assumption that $3k<\theta+3$, we have 
   \begin{align}\label{e-e2}
       \lambda_1(H(g))\geq \frac{1}{18}h(H(g))^2\geq \frac{\alpha_0^2}{2(\theta+3)^2}.
   \end{align}
From \eqref{e-e1} and \eqref{e-e2}, we have $\{H(g)\}_{g\geq 2}$ satisfy the assumption $(iii)$.
   
   For the remaining part, assume $g_m\leq g<g_{m+1}$ for some $m\geq m_0$. Since
   $$\frac{(3k-\theta)m-\theta}{1+\theta}-1\leq t_m+u\leq \frac{(3k-\theta)m-\theta}{1+\theta}+4,$$
   it follows that
   \begin{align*}
   \frac{n(g)}{g}-\theta&=\frac{3mk-t_m-u}{m+1+t_m+u}-\theta\\
   &\geq\frac{3mk-\frac{(3k-\theta)m-\theta}{1+\theta}-4}{m+1+\frac{(3k-\theta)m-\theta}{1+\theta}+4}-\theta\\
   &\geq-\frac{4(1+\theta)^2}{m(3k+1)+1}
   \end{align*}
   and 
   \begin{align*}
      \frac{n(g)}{g}-\theta&=\frac{3mk-t_m-u}{m+1+t_m+u}-\theta\\
      &\leq\frac{3mk-\left(\frac{(3k-\theta)m-\theta}{1+\theta}-1\right)}{m+1+\left(\frac{(3k-\theta)m-\theta}{1+\theta}-1\right)}-\theta \\
      &\leq \frac{(\theta+1)^2}{m(3k+1)-\theta}.
   \end{align*}
   In summary, we have
   \[
       \left|\frac{n(g)}{g}-\theta\right|\leq\max\left\{\frac{(\theta+1)^2}{m(3k+1)-\theta},\ \frac{4(1+\theta)^2}{m(3k+1)+1}\right\}.
   \]
   Since \(m\) tends to infinity as \(g\to\infty\), it follows that
   \begin{align*}
       \lim\limits_{g\to\infty}\frac{n(g)}{g}=\theta.
   \end{align*}
   and the function $n$ satisfies the assumption $(ii)$. 
   
   Hence the function $n$ and graphs $\{H(g)\}_{g\geq 2}$ are desired. The proof is complete.
\end{proof}

\subsection{From graph expanders to hyperbolic surfaces}\label{sec:surfaces}

In this subsection, we complete the proof of Theorem \ref{t-main-4}.

 We first recall some results of noncompact hyperbolic surfaces. For a noncompact hyperbolic surface $X$, the Rayleigh quotient is defined as
\begin{align}\label{e-rayq}
\RayQ(X)\colonequals\inf\limits_{\substack{ f\in H^1(X)\setminus\{0\}\\\int_{X}f=0}}\frac{\int_X |\nabla f|^2}{\int_X f^2}.
\end{align}
Similar to the case of compact surfaces, the following Cheeger's inequality for complete finite-area hyperbolic surfaces still holds, see e.g. \cite[Page 228]{Bu82}:
\begin{align}\label{e-inq}
\RayQ(X)\geq\frac{1}{4}h(X)^2.
\end{align}
For the existence of a nonzero eigenvalue, we have the following theorem of Reed--Simon \cite[Theorem XIII.1]{BM-book}.
\begin{theorem}\label{l-spec}
    Let $X$ be a noncompact hyperbolic surface of finite area. If $$\RayQ(X)<\frac{1}{4},$$
    then $X$ has a nonzero first eigenvalue $\lambda_1(X)$ with $\lambda_1(X)=\RayQ(X)$.
\end{theorem}

\begin{proof}[Proof of Theorem \ref{t-main-4}]
For a given \(\theta>0\), let \(\{H(g)\}_{g\geq 2}\) be the graph sequence constructed in Proposition \ref{t-main-3}. Then we have for $g\geq 2$,
\[
    h(H(g))\geq \frac{3\alpha_0}{\theta+3}
    \quad\text{and}\quad
    H(g)\in\mathcal{F}_{g,n(g)}.
\]
Moreover, the following equality holds, $$\lim\limits_{g\to\infty}\frac{n(g)}{g}=\theta.$$ Fix \(a>0\), and let \(X_a(H(g))\) be the complete finite-area hyperbolic surface obtained from \(H(g)\) by the pants construction of Section \ref{sec:graph-to-surface}. Then \(X_a(H(g))\in\mathcal{M}_{g,n(g)}\).

By Lemma \ref{l-com-h}, there exists a universal constant \(C=C(a)>0\) such that
\[
    h(X_a(H(g)))\geq C\min\{h(H(g)),1\}
    \geq C\min\left\{\frac{3\alpha_0}{\theta+3},1\right\},
\]
which together with \eqref{e-inq} implies that
\[
    \RayQ(X_a(H(g)))\geq \frac14 h(X_a(H(g)))^2
    \geq \min\left\{\frac{9C^2\alpha_0^2}{4(\theta+3)^2},\frac{C^2}{4}\right\}.
\]
Choose
\[
    c_\theta=
    \min\left\{\frac14,\frac{9C^2\alpha_0^2}{4(\theta+3)^2},\frac{C^2}{4}\right\}.
\]
Then from Theorem \ref{l-spec}, one may deduce that \(\Delta_{X_a(H(g))}\) has no eigenvalue contained in \((0,c_\theta)\), i.e.
\[
    \Spec(\Delta_{X_a(H(g))})\cap(0,c_\theta)=\emptyset.
\]
 Since \((\theta+3)^2\leq 9(1+\theta)^2\) for \(\theta>0\), we have $c_\theta\geq c_0(1+\theta)^{-2}$ for some universal constant $c_0>0$
 (after possibly decreasing it). Setting \(S_{g,n(g)}=X_a(H(g))\) completes the proof.
\end{proof}

\subsection{The case that \texorpdfstring{$\theta$}{theta} is small}
We thank Yunhui Wu for helpful discussions on the following proposition. 
\begin{proposition}\label{p-theat-small}
    For any \(\varepsilon>0\), there exists \(\theta_0>0\) such that for any \(0<\theta<\theta_0\), there exists a sequence of hyperbolic surfaces \(\{S_{g,n(g)}\}_{g\geq 2}\) such that
    \[
        \lim\limits_{g\to\infty}\frac{n(g)}{g}=\theta
    \]
    and for $g$ large enough,
    \[
        \Spec(\Delta_{S_{g,n(g)}})\cap\left(0,\frac{1}{4}-\varepsilon\right)=\emptyset.
    \]
\end{proposition}
Recall the notation \(\RayQ(X)\) from \eqref{e-rayq}. Then we have the following lemma.
\begin{lemma}\cite[Lemma 1.3]{BM01}\label{l-compare}
    For all \(r\) sufficiently large and for all \(\varepsilon > 0\), there exists \(C(r, \varepsilon) > 0\) with the following properties. Let \(S^C\) be a hyperbolic Riemann surface and \(\{x_i\}\) a collection of points of \(S^C\) such that the injectivity radii at the points \(x_i\) are all \(\geq r\) and the geodesic balls \(B(x_i, r)\) of radius \(r\) about the \(x_i\) are all disjoint. Let \(S^O\) be the surface obtained from \(S^C\) by replacing the points \(x_i\) by punctures. Then
\[
    \RayQ\left(S^O\right)\geq C(r,\varepsilon)\min\left\{\frac{1}{4}-\varepsilon,\lambda_1\left(S^C\right)\right\}.
\]
Furthermore, \(C(r,\varepsilon)\to 1\) as \(r\to\infty\).
\end{lemma}

\begin{proof}[Proof of Proposition \ref{p-theat-small}]
    For any hyperbolic surface \(X\) and point \(p\in X\), denote by \(\inj_X(p)\) the injectivity radius of \(p\). For any \(R>0\), also denote by
    \[
        X(R)=\{p\in X;\ \inj_X(p)\geq R\}.
    \]
    
    Let \(\mathrm{Prob}_{\mathrm{WP}}^g\) denote the probability measure on the moduli space \(\mathcal{M}_g\) obtained by normalizing the Weil-Petersson volume measure (see e.g. \cite{Mir13}).
     The proof of \cite[Theorem 4.5]{Mir13} shows that for any $\delta>0$,
     \begin{align}\label{e-wp-1}
         \lim\limits_{g\to\infty}\mathrm{Prob}_{\mathrm{WP}}^g\left(X\in\mathcal{M}_g;\ \Vol\left(X\left(\frac{1}{6}\log g\right)\right)\geq (1-\delta)\Vol(X)\right)=1.
     \end{align}
     From the recent near-optimal spectral gap results (see \cite[Theorem 1.6]{AM25} or \cite[Theorem 1.1]{HMT25}), we have for any $\delta>0$,
     \begin{align}\label{e-wp-2}
         \lim\limits_{g\to\infty}\mathrm{Prob}_{\mathrm{WP}}^g\left(X\in\mathcal{M}_g;\ \lambda_1(X)>\frac{1}{4}-\delta\right)=1.
     \end{align}
     Combining \eqref{e-wp-1} and \eqref{e-wp-2}, one may conclude that for any \(0<\delta<\frac{1}{4}\) and all sufficiently large \(g\), there exists a hyperbolic surface \(X_g\in\mathcal{M}_g\) such that
    \begin{enumerate}
        \item \(\lambda_1(X_g)>\frac{1}{4}-\delta\);
        \item \(\Vol\left(X_g\left(\frac{1}{6}\log g\right)\right)\geq (1-\delta)\Vol(X_g)>2\pi(g-1)\).
    \end{enumerate}
    For any \(p\in X_g\) and \(r>0\), assume \(B(p,r)\) is the geodesic ball with center \(p\) and radius \(r\). Set \(n(g)=\left\lfloor\theta g\right\rfloor\). We claim that there exist points
    \[
        \{p_1,\ldots,p_{n(g)}\}\subseteq X_g\left(\frac{1}{6}\log g\right)
    \]
    such that
    \begin{align}\label{e-disjoint}
    B\left(p_i,\frac{1}{2}\log\frac{1}{4\theta}\right)\cap B\left(p_j,\frac{1}{2}\log\frac{1}{4\theta}\right)=\emptyset\quad\text{for }1\leq i\neq j\leq n(g).
    \end{align}
    In fact, assume points \(\{p_1,\ldots,p_k\}\subseteq X_g\left(\frac{1}{6}\log g\right)\) satisfy \eqref{e-disjoint} for some \(k<n(g)\). Since 
    \begin{align*}
\sum\limits_{i=1}^k\Vol\left(B\left(p_i,\log\frac{1}{4\theta}\right)\right)&\leq k\cdot 4\pi\left(\cosh\log\frac{1}{4\theta}-1\right)\\
&< \theta g\cdot 4\pi\cdot\frac{1}{4\theta}<2\pi(g-1),
\end{align*}
it follows that there exists a point
\[
    p_{k+1}\in X_g\left(\frac{1}{6}\log g\right)\setminus\bigcup\limits_{i=1}^k B\left(p_i,\log\frac{1}{4\theta}\right)
\]
such that
\[
B\left(p_{k+1},\frac{1}{2}\log\frac{1}{4\theta}\right)\cap B\left(p_i,\frac{1}{2}\log\frac{1}{4\theta}\right)=\emptyset\quad\text{for }1\leq i\leq k.
\]
Hence the desired points \(\{p_1,\ldots,p_{n(g)}\}\) exist. Denote by \(X_{g,n(g)}\) the hyperbolic surface obtained by replacing \(p_1,\ldots,p_{n(g)}\) with punctures. Then from Lemma \ref{l-compare} and \eqref{e-disjoint}, we have for any \(0<\delta<\frac{1}{2}\),
\begin{align*}
\RayQ(X_{g,n(g)})&>C\left(\frac{1}{2}\log\frac{1}{4\theta},\delta\right)\min\left\{\frac{1}{4}-\delta,\lambda_1(X_g)\right\}\\
&\geq C\left(\frac{1}{2}\log\frac{1}{4\theta},\delta\right)\left(\frac{1}{4}-\delta\right),
\end{align*}
where \(C\left(\frac{1}{2}\log\frac{1}{4\theta},\delta\right)\) is the constant in Lemma \ref{l-compare}.

It remains to choose the parameters. We may assume \(0<\varepsilon<\frac{1}{4}\), since otherwise the interval \((0,\frac{1}{4}-\varepsilon)\) is empty. Take \(0<\delta<\varepsilon\), then we have $$\frac{1}{4}-\delta>\frac{1}{4}-\varepsilon.$$ 
For such fixed $\delta$, Lemma \ref{l-compare} gives
$$\lim\limits_{\theta\to 0}C\left(\frac{1}{2}\log\frac{1}{4\theta},\delta\right)=1,$$
which implies that there exists \(\theta_0>0\) such that for every \(0<\theta<\theta_0\),
\[
    C\left(\frac{1}{2}\log\frac{1}{4\theta},\delta\right)\left(\frac{1}{4}-\delta\right)>\frac14-\varepsilon.
\]
Therefore \(\RayQ(X_{g,n(g)})>\frac{1}{4}-\varepsilon\) for all sufficiently large \(g\), together with Theorem \ref{l-spec}, one may deduce that \(\Delta_{X_{g,n(g)}}\) has no eigenvalue contained in \((0,\frac{1}{4}-\varepsilon)\), i.e.
\[
    \Spec(\Delta_{X_{g,n(g)}})\cap(0,\frac14-\varepsilon)=\emptyset.
\]
Setting \(S_{g,n(g)}=X_{g,n(g)}\) completes the proof.
\end{proof}

\section{The Steklov upper bound and the large-cusp obstruction}\label{sec:steklov}
In this section, we prove Theorem \ref{t-main-1}, which gives a deterministic obstruction when the number of boundary vertices is too large compared with the topological genus.

We first recall the statement of Theorem \ref{t-main-1} for convenience and we always assume that $n\geq 2$ in this section.

\medskip
\noindent\textbf{Theorem \ref{t-main-1}.}
For integers $n\geq 2,\ g\geq 0$, and connected graph \(G\in\mathcal{F}_{g,n}\), we have
\[
    \lambda_1(G)\leq\sigma_1(G)\leq\frac{16(g+1)}{3n}.
\]
\medskip

We next prove the following two lemmas. For a graph $G=(V,E)$ and a subset of edges $K\subseteq E$, we write $G\setminus K$ for a graph $(V,E\setminus K)$ with edges in $K$ removed. For two graphs $G_i=(V_i,E_i)$, $i=1,2$, we write $G_1\subseteq G_2$ if $V_1\subseteq V_2$ and $E_1\subseteq E_2$.  
\begin{lemma}\label{l-div}
With the same assumptions as in Theorem \ref{t-main-1}, there exist $g+1$ edges $$\{e_1,...,e_{g+1}\}\subseteq E(G)$$ such that 
$$G\setminus\{e_1,...,e_{g+1}\}=T_1\sqcup T_2$$
where $T_1$ and $T_2$ are two disjoint trees.
\end{lemma}
\begin{proof}
   Recall that the Euler characteristic $\chi(G)$ of graph $G$ is defined as 
   $$\chi(G)=|V(G)|-|E(G)|.$$
   It is well-known that for any connected graph $G$,
   $$\chi(G)\leq 1.$$
   Moreover, $\chi(G)=1$ if and only if $G$ is a tree. 
Therefore, for any connected graph $G\in\mathcal{F}_{g,n}$, we have
   $$\chi(G)=2g-2+2n-\frac{6g-6+4n}{2}=1-g.$$
   
   If $g=0$, then $G$ is a tree, the result is trivial.

   If $g\geq 1$, then $$\chi(G)=1-g<1,$$ 
   which implies that $G$ is not a tree. It follows that there exists an edge $e_1\in E(G)$ such that $G\setminus\{e_1\}$ is connected. Denote $G_1=G\setminus\{e_1\}$. Then $G_1$ is a connected graph and
   $$\chi(G_1)=|V(G)|-(|E(G)|-1)=2-g.$$
   Repeating the above procedure for $g$ times, one may obtain a set of edges 
   $$\{e_1,...,e_g\}\subseteq E(G)$$
   and a sequence of graphs $G_g\subseteq G_{g-1}\subseteq...\subseteq G_1\subseteq G_0=G$ such that 
   \begin{enumerate}
       \item $G_{i}=G_{i-1}\setminus\{e_i\}$ for all $1\leq i\leq g$;
       \item $G_i$ is connected and $\chi(G_i)=i+1-g$ for all $0 \leq i\leq g$.
   \end{enumerate}
   Then $\chi(G_g)=1$, which implies that the graph $G_g$ is a tree. There exists an edge $$e_{g+1}\in E(G_g)\subseteq E(G)$$ such that $G_g\setminus\{e_{g+1}\}$ consists of two disjoint  trees. Hence the edges
   $$\{e_1,...,e_{g+1}\}$$
   are desired and this completes the proof.
\end{proof}
Assume $G\in\mathcal{F}_{g,n}$ is a connected graph, denote by $\delta G$ the set consisting of all vertices with degree $1$ in $V(G)$. Also recall that for any subset $\Omega\subseteq V(G)$, $\partial\Omega=E(\Omega,\Omega^c)$ (see Subsection \ref{sec-ch}). Then we have
\begin{lemma}\label{l-sub}
    With the same assumptions in Theorem \ref{t-main-1}, there exists a subset $H$ of $V(G)$ such that
    \begin{enumerate}[label=(\roman*)]
        \item $|\partial H|\leq g+1$;
        \item $\frac{n}{4}\leq \left|H\cap\delta G\right|\leq\frac{n}{2}$.
    \end{enumerate}
\end{lemma}
\begin{proof}
    
    We prove the lemma by a contradiction argument.
    
    \textbf{Assumption} $(\star)$: there is no  subset $H$ of $V(G)$ satisfying the conditions $(i)$ and $(ii)$.
    
    From Lemma \ref{l-div}, there exists a set of edges $\{e_1,...,e_{g+1}\}\subseteq E(G)$ such that
    $$G\setminus\{e_1,...,e_{g+1}\}=T_1\sqcup T_1^\prime$$
    where $T_1$ and $T_1^\prime$ are disjoint trees. From assumption $(\star)$, one may assume 
    $$|V(T_1)\cap\delta G|>\frac{n}{2}\text{ and }|\partial V(T_1)|\leq g+1.$$ 
    Now we prove that there exists a subtree $T_2\subseteq T_1$ such that
    \begin{align}\label{e-subg}
    |V(T_2)\cap\delta G|>\frac{n}{2},\ |\partial V(T_2)|\leq g+1\text{ and }| V(T_2)|<| V(T_1)|.
    \end{align}
    Take an edge $e_i\in \partial T_1$ such that $e_i=w\sim v$ with $w\in V(T_1)$ and $v\in V(T_1^\prime)$.
\begin{figure}[ht]
\centering
\includegraphics[width=0.8\textwidth]{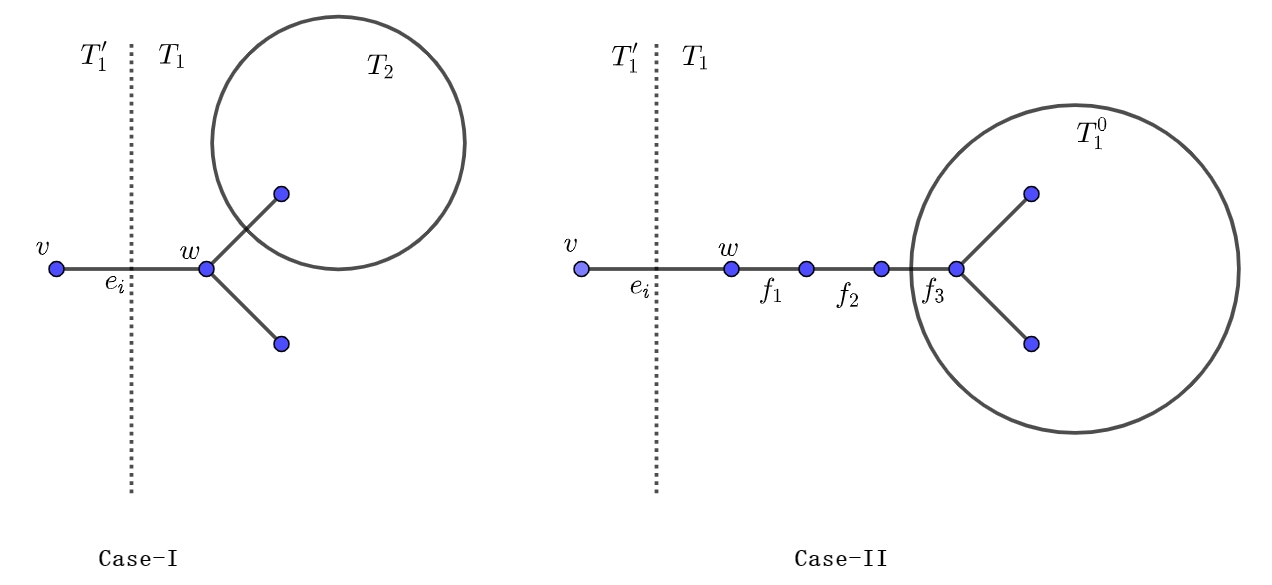}
\caption{Two Cases of edges adjacent to $w$ in $T_1$.}
\label{}
\end{figure}

    \underline{Case I. There are two edges in $T_1$ containing $w$.} There are two connected components after removing such two edges from $T_1$. Denote by $T_2$ the connected component such that
    $$|V(T_2)\cap\delta G|\geq\frac{1}{2}|V(T_1)\cap\delta G|>\frac{n}{4}.$$
   It is clear that $T_2$ is a subtree of $T_1$ such that  
   $$|\partial V(T_2)|\leq |\partial V(T_1)|\leq g+1\text{ and }| V(T_2)|<| V(T_1)|.$$
  It follows from assumption $(\star)$ that  $$|V(T_2)\cap\delta G|>\frac{n}{2}$$ and $T_2$ is a subtree of $T_1$ satisfying condition \eqref{e-subg}.

    \underline{Case II. There is only one edge in $T_1$ containing $w$.} After removing a sequence of consecutive edges $e_i,\ f_1,...,\ f_k$ from $T_1$, one may obtain a subtree $T_1^0\subseteq T_1$ such that $f_k\in\partial V(T_1^0)$ and there are two edges adjacent to $f_k$ in $T_1^0$, moreover
    $$\left|V\left(T_1^0\right)\cap\delta G\right|=|V(T_1)\cap\delta G|>\frac{n}{2}\text{ and }\left|\partial V\left(T_1^0\right)\right|\leq g+1.$$
     By the same argument as in Case I, one may obtain the desired subtree $T_2$. 

    Repeating the procedure above,  one may conclude that there exists an infinite sequence of trees $T_1\supseteq T_2\supseteq...\supseteq T_k\ ...$ such that
    \begin{enumerate}
    \item $| V(T_i)|>| V(T_{i+1})|$ for all $i\geq 1$;
    \item $|V(T_i)\cap\delta G|>\frac{n}{2}$ and $|\partial V(T_i)|\leq g+1$ for all $i\geq 1$.
    \end{enumerate}
    This yields a contradiction since $G$ is a finite graph. The proof is complete.
\end{proof}
Now we are ready to prove Theorem \ref{t-main-1}.

\begin{proof}[Proof of Theorem \ref{t-main-1}]
    Let $H$ be the subset obtained in Lemma \ref{l-sub}. Define the function $f:V(G)\to\mathbb{R}$ as follows:
$$f(x)=\begin{cases}1-\frac{|H\cap\delta G|}{n},&\text{ if }~x\in H,\\
-\frac{|H\cap\delta G|}{n},&\text{ if }~x\notin H.\end{cases}$$
Denote by $s=|H\cap \delta G|$. Then we have 
\begin{align*}
\sum\limits_{x\in\delta G}f(x)=0
\end{align*}
and
\begin{align*}
    \sum\limits_{x\in\delta G}f^2(x)=s\left(1-\frac{s}{n}\right)^2+(n-s)\left(\frac{s}{n}\right)^2=\frac{s(n-s)}{n}
\end{align*}
It follows from Lemma \ref{l-sub} that
\begin{align*}
    R(f)&=\frac{\sum\limits_{(x,y)\in E(G)}(f(x)-f(y))^2}{\sum\limits_{x\in\delta G}f^2(x)}\\
    &=\frac{n|\partial H|}{s(n-s)}\leq \frac{16(g+1)}{3n}.
\end{align*}
Together with  \eqref{e-stek} and Theorem \ref{t-com}, we have
\begin{align*}
\lambda_1(G)\leq \sigma_1(G)\leq R(f)\leq\frac{16(g+1)}{3n}.
\end{align*}
The proof is complete.
\end{proof}

As a direct corollary, we obtain:
\begin{corollary}\label{t-small}
    Assume that $n(g)$ satisfies $\lim\limits_{g\to\infty}\frac{n(g)}{g}=\infty$ and $\{G_g\}_{g\geq 1}$ is a sequence of connected graphs such that $G_{g}\in\mathcal{F}_{g,n(g)}$, then 
    $$\lim\limits_{g\to\infty}\lambda_1(G_g)=\lim\limits_{g\to\infty}\sigma_1(G_g)=0.$$
\end{corollary}

\noindent{\bf Acknowledgments.}
The authors would like to thank Yunhui Wu and Will Hide for valuable comments and helpful suggestions. The authors  would also like to thank Michael Magee for bringing our attention to some recent developments in spectral gaps on hyperbolic surfaces. The second author is supported by NSFC, No. 12371056.
The third author is supported by NSFC, No. 12401081.

\noindent{\bf  Keywords}: \quad Hyperbolic surfaces, spectral gaps, Cheeger constants, pants decompositions, expander graphs. 

\noindent\small {\bf \small  Declarations of interest}: none.

\noindent{\bf \small Data availability statement:} There are no new data associated with this article.

\bibliographystyle{plain}
\bibliography{ref}
 
\noindent {Qi Guo\\
School of Mathematics,\\
Renmin University of China, Beijing, 100872, P.R. China\\
e-mail: qguo@ruc.edu.cn}
\medskip
\\
\noindent {Bobo Hua\\
School of Mathematical Sciences, LMNS,\\
Fudan University, Shanghai 200433, P.R. China\\
e-mail: bobohua@fudan.edu.cn }
\medskip
\\
\noindent {Yang Shen\\
School of Mathematics, \\
Hunan University, Changsha, 410082, P.R. China\\
e-mail: shen-y628@hnu.edu.cn }

\end{document}